\documentclass[12pt, a4paper, oneside]{amsart}
\usepackage{amsmath, amsthm, amssymb, geometry, hyperref, tikz-cd, mathrsfs}
\usepackage[shortlabels, inline]{enumitem}

\newtheorem{theorem}{Theorem}[section]

\newtheorem{lemma}[theorem]{Lemma}
\newtheorem{corollary}[theorem]{Corollary}

\theoremstyle{definition}
\newtheorem{definition}[theorem]{Definition}

\newtheorem{example}[theorem]{Example}
\newtheorem{remark}[theorem]{Remark}

\newcommand{\bA}{\mathbb{A}}

\newcommand{\bC}{\mathbb{C}}

\newcommand{\bG}{\mathbb{G}}

\newcommand{\bN}{\mathbb{N}}

\newcommand{\sA}{\mathscr{A}}

\newcommand{\Ell}{\mathrm{Ell}}

\newcommand{\id}{\mathrm{id}}

\DeclareMathOperator{\Sing}{Sing}

\begin{document}

\title{Surjective morphisms onto subelliptic varieties}
\author{Yuta Kusakabe}
\address{Department of Mathematics, Graduate School of Science, Kyoto University, Kyoto 606-8502, Japan}
\email{kusakabe@math.kyoto-u.ac.jp}


\begin{abstract}
    We prove that every smooth subelliptic variety admits a surjective morphism from an affine space.
    This result gives partial answers to the questions of Arzhantsev and Forstneri\v{c}.
    As an application, we characterize open images of morphisms between affine spaces.
    We also obtain the jet interpolation theorem for morphisms from zero-dimensional subschemes of affine varieties to smooth subelliptic varieties.
\end{abstract}

\maketitle

%
%

\section{Introduction}
\label{section:introduction}

In the context of Oka theory, Gromov \cite{Gromov1986, Gromov1989} introduced several ellipticity conditions for complex manifolds and algebraic varieties (cf. \cite{Forstneric2017,Forstnerich}).
Roughly speaking, subellipticity for algebraic varieties means the existence of many dominant morphisms from affine spaces $\bA^n$.
The following Gromov's condition $\Ell_1$ is one of the characterizations of smooth subelliptic varieties (Definition \ref{definition:ellipticity}).

\begin{definition}
    \label{definition:Ell_1}
    A smooth variety\footnote{Throughout this paper, for simplicity, we work in the category of (irreducible) algebraic varieties over $\bC$ unless otherwise stated.} $Y$ satisfies $\Ell_1$ if for any affine variety $X$ and any morphism $f:X\to Y$ there exist $n\in\bN$ and a morphism $s:X\times\bA^n\to Y$ such that $s(x,0)=f(x)$ and $s(x,\cdot):\bA^n\to Y$ is smooth at $0\in\bA^n$ for each $x\in X$.
\end{definition}

The most typical examples are smooth locally flexible varieties (Definition \ref{definition:flexibility}).
This class includes smooth nondegenerate toric varieties, varieties covered by affine spaces and certain homogeneous spaces (Example \ref{example:flexibility}).

If we replace in Definition \ref{definition:Ell_1} morphisms from affine varieties by holomorphic maps from Stein spaces, we obtain the characterization of Oka manifolds \cite{Kusakabe2021a, Kusakabe2021b}.
By Gromov's Oka principle \cite{Forstneric2002a,Forstneric2010,Gromov1989}, the analytification of a smooth subelliptic variety is Oka.
Forstneri\v{c} \cite{Forstneric2017a} proved every (connected) Oka manifold admits a surjective holomorphic map from an affine space.
Thus, it is natural to ask whether every smooth subelliptic variety admits a surjective morphism from an affine space.

In this paper, we give the following answer to the above problem.

\begin{theorem}
    \label{theorem:surjective}
    For any smooth subelliptic variety $Y$, there exists a surjective morphism $f:\bA^{\dim Y+1}\to Y$ such that $f(\bA^{\dim Y+1}\setminus\Sing(f))=Y$ where $\Sing(f)=\{x\in\bA^{\dim Y+1}:\mbox{$f$ {\rm is not smooth at} $x$}\}$ is the singular locus of $f$.
\end{theorem}

In fact, Forstneri\v{c} \cite[Theorem 1.6]{Forstneric2017a} also proved that every smooth \emph{proper} subelliptic variety $Y$ admits a morphism $f:\bA^{\dim Y}\to Y$ such that $f(\bA^{\dim Y}\setminus\Sing(f))=Y$ and asked whether one can omit the properness assumption.
Theorem \ref{theorem:surjective} gives an affirmative answer if the dimension of an affine space is ignored.
Motivated by Forstneri\v{c}'s result, Arzhantsev \cite[Proposition 2]{Arzhantsevb} showed that every \emph{very flexible} variety can be realized as the image of an affine space and asked to find a sufficient condition to be such an image \cite[Problem 1]{Arzhantsevb}.
Since every very flexible variety is smooth and subelliptic \cite[Corollary 1.11]{Arzhantsev2013}, Theorem \ref{theorem:surjective} generalizes Arzhantsev's result and gives an answer to his problem.
In \cite{Kusakabe2021}, we proved that every smooth proper locally flexible variety admits a surjective \emph{holomorphic submersion} from an affine space.

Since the complement of a Zariski closed subset of codimension at least two in a smooth locally flexible variety is locally flexible \cite{Flenner2016}, the following corollary holds.
Note that it has already been proved for smooth flexible quasi-affine varieties by Arzhantsev \cite[Theorem 1]{Arzhantsevb}.

\begin{corollary}
    Let $Y$ be a smooth locally flexible variety and $U\subset Y$ be a Zariski open subset.
    If the complement $Y\setminus U$ is of codimension at least two, then $U$ is the image of a morphism from an affine space.
\end{corollary}

In the case of an affine space, if a Zariski closed subset $Z\subsetneq\bA^n$ is not of codimension at least two, the complement $\bA^n\setminus Z$ admits a nonconstant invertible regular function.
Thus we obtain the following characterization of open images of morphisms between affine spaces.
This of course follows also from Arzhantsev's result \cite[Theorem 1]{Arzhantsevb}.

\begin{corollary}
    For a Zariski open subset $U\subset\bA^n$, the following are equivalent:
    \begin{enumerate}
        \item $U$ is the image of a morphism from an affine space.
        \item The complement $\bA^n\setminus U$ is of codimension at least two.
    \end{enumerate}
\end{corollary}

One of the main features of Oka manifolds is that the interpolation theorem holds for holomorphic maps from Stein spaces.
More precisely, for any closed analytic subset $Z\subset X$ of a Stein space $X$ and any holomorphic map $f:Z\to Y$ to an Oka manifold $Y$, there exists a holomorphic extension $X\to Y$ of $f$ if it admits a continuous extension $X\to Y$.
On the other hand, L\'{a}russon and Truong proved that most smooth subelliptic varieties (e.g. smooth proper subelliptic varieties) do not have the algebraic version of this interpolation property \cite[Theorem 2]{Larusson2019}.
However, since every smooth subelliptic variety is the image of an affine space by Theorem \ref{theorem:surjective}, we can extend morphisms from finite points of affine varieties by lifting to $\bA^n$.
Furthermore, since morphisms from zero-dimensional subschemes lift along smooth morphisms, we obtain the following jet interpolation theorem.

\begin{corollary}
    \label{corollary:interpolation}
    Let $Y$ be a smooth subelliptic variety, $X$ be an affine variety and $Z\subset X$ be a zero-dimensional subscheme of $X$.
    Then for any morphism $f:Z\to Y$ there exists a morphism $\tilde f:X\to Y$ such that $\tilde f|_Z=f$.
\end{corollary}

%
%

\section{Sprays and subellipticity}
\label{section:subellipticity}

In order to prove Theorem \ref{theorem:surjective}, let us recall the notion of sprays.

\begin{definition}
    \label{definition:ellipticity}
    Let $X$ be a variety and $Y$ be a smooth variety.
    \begin{enumerate}[leftmargin=*]
        \item A \emph{spray} over a morphism $f:X\to Y$ is a triple $(E,p,s)$ where $p:E\to X$ is a vector bundle and $s:E\to Y$ is a morphism such that $s(0_{x})=f(x)$ for each $x\in X$.
        \item A family of sprays $\{(E_j,p_j,s_j)\}_{j=1}^n$ over a morphism $f:X\to Y$ is \emph{dominating} if  the differentials $ds_j:TE_j\to TY$ $(j=1,\ldots,n)$ satisfy
        \[
            \sum_{j=1}^n ds_j\left(T_{0_x}p_j^{-1}(x)\right)=T_{f(x)}Y
        \]
        for each $x\in X$.
        \item $Y$ is \emph{elliptic} (resp. \emph{subelliptic}) if there exists a dominating spray (resp. a dominating family of sprays) over the identity morphism $\id_Y$ of $Y$.
    \end{enumerate}
\end{definition}

By Serre's theorem A \cite[p.\,237, Th\'{e}or\`{e}me 2]{Serre1955a}, a smooth variety $Y$ satisfies $\Ell_1$ (Definition \ref{definition:Ell_1}) if and only if for any affine variety $X$ and any morphism $f:X\to Y$ there exists a dominating spray over $f$.
The following Gromov's localization theorem establishes the equivalence between $\Ell_1$ and subellipticity (see \cite[Proposition 8.8.11]{Forstneric2017}).

\begin{theorem}[{Gromov \cite[3.5.B]{Gromov1989}, cf. \cite[Proposition 1.4]{Kaliman2018}}]
    \label{theorem:localization}
    Let $Y$ be a smooth variety, $U\subset Y$ be a Zariski open subset and $(E,p,s)$ be a spray over the inclusion morphism $U\hookrightarrow Y$.
    Then for each $y\in U$ there exists a spray $(\widetilde E,\tilde p,\tilde s)$ over the identity morphism $\id_Y$ such that $ds(T_{0_y}p^{-1}(y))=d\tilde s(T_{0_y}\tilde p^{-1}(y))$.
    In particular, every smooth variety covered by subelliptic Zariski open subvarieties is subelliptic.
\end{theorem}

An algebraic $\bG_a$-action $s:Y\times\bG_a\to Y$ on a smooth variety $Y$ is an example of a spray over the identity morphism $\id_Y$.
By using sprays of this type, subelliptic varieties called flexible varieties are defined as follows.

\begin{definition}[{cf. \cite{Arzhantsev2013,Arzhantsev2013a,Kaliman2018}}]
    \phantomsection\label{definition:flexibility}
    \begin{enumerate}[leftmargin=*]
        \item A smooth variety $Y$ is \emph{flexible} if for each $y\in Y$ the tangent space $T_{y}Y$ is spanned by the tangent vectors to the orbits of $\bG_a$-actions on $Y$ through $y$.
        \item A smooth variety is \emph{locally flexible} if it is covered by flexible \emph{quasi-affine} Zariski open subvarieties.
    \end{enumerate}
\end{definition}

Note that Theorem \ref{theorem:localization} yields subellipticity of smooth locally flexible varieties.

\begin{example}
    \phantomsection\label{example:flexibility}
    \begin{enumerate}[leftmargin=*]
        \item \label{item:A-covered}
        The affine spaces $\bA^n$ $(n\in\bN)$ are obviously flexible.
        Therefore, every variety covered by affine spaces is locally flexible.
        Such a variety is said to be of class $\sA_0$ \cite[\S 6.4]{Forstneric2017} or $A$-covered \cite{Arzhantsev2014}.
        \item \label{item:complement}
        By the Flenner--Kaliman--Zaidenberg theorem \cite[Theorem 1.1]{Flenner2016}, for any smooth locally flexible variety $Y$ and any Zariski closed subset $Z\subset Y$ of codimension at least two, the complement $Y\setminus Z$ is also locally flexible.
        \item \label{item:toric}
        A toric variety is said to be \emph{nondegenerate} if it has no torus factor.
        Arzhantsev, Zaidenberg and Kuyumzhiyan proved that the smooth locus of a nondegenerate affine toric variety is flexible \cite[Theorem 0.2]{Arzhantsev2012}.
        By using this result and (\ref{item:complement}), L\'{a}russon and Truong showed that every smooth nondegenerate toric variety is locally flexible \cite[Theorem 3]{Larusson2019}.
        \item \label{item:homogeneous}
        Let $G$ be a connected linear algebraic group and $H$ be a closed subgroup of $G$.
        Then the homogeneous space $G/H$ is flexible if and only if there is no nonconstant morphism $G/H\to\bA^1\setminus\{0\}$ \cite[Proof of Theorem C]{Arzhantsevb} (see also \cite[Corollary 1.11]{Arzhantsev2013}).
    \end{enumerate}
\end{example}

\begin{remark}
    Arzhantsev \cite{Arzhantsevb} proved that smooth locally flexible varieties in Example \ref{example:flexibility} (\ref{item:A-covered}), (\ref{item:toric}) and (\ref{item:homogeneous}) admit surjective morphisms from affine spaces.
    At present, it is not known whether there exists a smooth subelliptic variety which is not locally flexible.
    For example, the blowup of a smooth locally flexible variety along a smooth center is subelliptic \cite{Kaliman2018,Larusson2017} but not known to be locally flexible (see also \cite[Corollary 1.5]{Kusakabe2020a}).
\end{remark}

In the rest of this section, we recall a few facts about composed sprays and pullback sprays defined as follows.

\begin{definition}[{cf. \cite[Definition 6.3.5]{Forstneric2017}}]
    \label{definition:composed}
    Let $X$ be a variety and $Y$ be a smooth variety.
    \begin{enumerate}[leftmargin=*]
    \item For a family of sprays $\{(E_{j},p_{j},s_{j})\}_{j=1}^{n}$ over the identity morphism $\id_{Y}$, the \emph{composed spray}\footnote{In general, the projection $p_{1}*\cdots*p_{n}:E_{1}*\cdots*E_{n}\to Y$ has no canonical vector bundle structure, but it admits a well-defined zero section.} $(E_{1}*\cdots*E_{n},p_{1}*\cdots*p_{n},s_{1}*\cdots*s_{n})$ is defined by
    \begin{align*}
        E_{1}*\cdots*E_{n}&=\left\{(e_{1},\ldots,e_{n})\in\prod_{j=1}^{n}E_{j}:s_{j}(e_{j})=p_{j+1}(e_{j+1}),\ j=1,\ldots,n-1\right\},\\
        (p_{1}*\cdots&*p_{n})(e_{1},\ldots,e_{n})=p_{1}(e_{1}),\quad (s_{1}*\cdots*s_{n})(e_{1},\ldots,e_{n})=s_{n}(e_{n}).
    \end{align*}
    \item For a morphism $f:X\to Y$ and a composed spray $(E,p,s)$ over the identity morphism $\id_Y$, the \emph{pullback spray} $(f^{*}E,f^{*}p,f^{*}s)$ over $f:X\to Y$ consists of the pullback bundle $f^{*}p:f^{*}E\to X$ of $p:E\to Y$ and the composition $f^{*}s:f^{*}E\to Y$ of the induced morphism $f^{*}E\to E$ and $s:E\to Y$.
    \end{enumerate}
\end{definition}

By the following lemma, subellipticity is equivalent to the existence of a \emph{dominating} composed spray.

\begin{lemma}[{cf. \cite[Lemma 6.3.6]{Forstneric2017}}]
    \label{lemma:composed}
    Let $Y$ be a smooth variety.
    Then a family of sprays $\{(E_j,p_j,s_j)\}_{j=1}^n$ over the identity morphism $\id_{Y}$ is dominating if and only if the morphism
    \[
        (p_{1}*\cdots*p_{n},s_{1}*\cdots*s_{n}):E_{1}*\cdots*E_{n}\to Y\times Y
    \]
    is smooth near the zero section $\{(0_{y},\cdots,0_{y})\in E_{1}*\cdots*E_{n}:y\in Y\}$.
\end{lemma}

By repeatedly using Serre's theorem A \cite{Serre1955a}, we can obtain the following lemma.

\begin{lemma}[{cf. \cite[Proof of Lemma 3.6]{Forstneric2006a}}]
    \label{lemma:Serre}
    Let $Y$ be a smooth variety and $(E,p,s)$ be the composed spray of a family of sprays over the identity morphism $\id_{Y}$.
    Then for any morphism $f:X\to Y$ from an affine variety $X$ there exist $n\in\bN$ and a surjective smooth morphism $\varphi:X\times\bA^n\to f^*E$ such that $f^{*}p\circ\varphi:X\times\bA^{n}\to X$ is the projection and the image $\varphi(X\times\{0\})$ is the zero section of $f^{*}E$.
\end{lemma}

%
%

\section{Proof of Theorem \ref{theorem:surjective}}
\label{section:surjective}

We first prove the existence of a morphism $h:\bA^{N}\to Y$ such that $h(\bA^{N}\setminus\Sing(h))=Y$ for large $N\gg 0$ and then reduce the dimension of an affine space.

\begin{proof}[Proof of Theorem \ref{theorem:surjective}]
    Let $Y$ be a smooth subelliptic variety.
    Since $Y$ satisfies $\Ell_1$ (Definition \ref{definition:Ell_1}), there exist $n\in\bN$ and a morphism $g:\bA^n\to Y$ which is smooth at $0\in\bA^n$.
    Since a smooth morphism is open and $0\in\bA^n\setminus\Sing(g)\neq\emptyset$, the image $U=g(\bA^n\setminus\Sing(g))\subset Y$ is a dense Zariski open subset.
    Let $(E,p,s)$ be the composed spray of a dominating family of sprays over the identity morphism $\id_Y$.
    By Lemma \ref{lemma:composed}, the singular locus $\Sing(p,s)$ of the morphism $(p,s):E\to Y\times Y$ is contained in the complement $E\setminus Z$ of the zero section $Z$ of $E$.
    Note that the image $(p,s)(Z)$ is the diagonal subset of $Y\times Y$.

    Take a point $y\in Y\setminus U$ arbitrarily.
    Since the morphism $(p,s):E\to Y\times Y$ is smooth at the zero $0_y\in E$ over $y$,
    \[
        \dim_{0_y}(p,s)^{-1}(Y\times\{y\})=\dim E-\dim Y.
    \]
    If we assume that
    \[
        (p,s)^{-1}(Y\times\{y\})\subset p^{-1}(Y\setminus U)\cup\Sing(p,s)\subset E,
    \]
    then
    \begin{align*}
        \dim_{0_y}(p,s)^{-1}(Y\times\{y\})&=\dim_{0_y}\left((p,s)|_{p^{-1}(Y\setminus U)}\right)^{-1}((Y\setminus U)\times\{y\}) \\
        &<\dim E-\dim Y
    \end{align*}
    holds since $(p,s)|_{p^{-1}(Y\setminus U)}:p^{-1}(Y\setminus U)\to (Y\setminus U)\times Y$ is also smooth at $0_y$, but this contradicts the first equality.
    Therefore
    \[
        y\in s\left(p^{-1}(U)\setminus\Sing(p,s)\right)\subset s\left(p^{-1}(U)\setminus\Sing(s)\right)
    \]
    holds and it follows that $Y\setminus U\subset s(p^{-1}(U)\setminus\Sing(s))$.

    Let us consider the pullback spray $(g^*E,g^*p,g^*s)$ of $(E,p,s)$ by the above morphism $g:\bA^n\to Y$.
    By Lemma \ref{lemma:Serre}, there exists a surjective smooth morphism $\varphi:\bA^n\times\bA^{m}\to g^{*}E$ over $\bA^n$ which preserves the zero section.
    Set $h=g^{*}s\circ\varphi:\bA^{n+m}\to Y$.
    Then it follows that
    \begin{align*}
        Y=U\cup(Y\setminus U)&\subset g\left(\bA^n\setminus\Sing(g)\right)\cup s\left(p^{-1}(U)\setminus\Sing(s)\right) \\
        &\subset h((\bA^n\times\{0\})\setminus\Sing(h))\cup h(((\bA^n\setminus\Sing(g))\times\bA^{m})\setminus\Sing(h)) \\
        &\subset h(\bA^{n+m}\setminus\Sing(h)).
    \end{align*}

    Set $d=\dim Y$.
    For each $y\in Y$, there exists a $d$-dimensional affine subspace $H_y\subset\bA^{n+m}$ such that $y\in h(H_y\setminus\Sing(h|_{H_y}))$.
    Since $Y$ is quasi-compact, there exist finitely many points $y_1,\ldots, y_l\in Y$ such that
    \[
        Y=\bigcup_{j=1}^l h\left(H_{y_j}\setminus\Sing\left(h|_{H_{y_j}}\right)\right).
    \]
    Choose pairwise distinct points $x_1,\ldots, x_l\in\bA^1$ and take a morphism $\psi:\bA^{d+1}\to\bA^{n+m}$ so that for each $j=1,\ldots,l$ the restriction $\psi|_{\bA^{d}\times\{x_j\}}$ is an affine transformation whose image is $H_{y_{j}}$.
    Then the morphism $f=h\circ\psi:\bA^{d+1}\to Y$ satisfies $f(\bA^{d+1}\setminus\Sing(f))=Y$.
\end{proof}

\begin{remark}
    As mentioned in the introduction, Forstneri\v{c} proved that every smooth proper subelliptic variety admits a surjective morphism from an affine space \cite[Theorem 1.6]{Forstneric2017a}.
    In his proof, an analytic tool called the \emph{homotopy Runge approximation theorem} (equivalent to subellipticity \cite{Larusson2019}) was used.
    Note that our proof of Theorem \ref{theorem:surjective} given above is purely algebro-geometric.
    Thus our main result also holds for smooth subelliptic varieties over more general fields.
    Elliptic algebraic geometry over general fields will be studied in future work.
\end{remark}

%
%

\section*{Acknowledgments}

This work was supported by JSPS KAKENHI Grant Number JP21K20324.

%
%


\end{document}